\documentclass[11pt]{amsart}
\usepackage{amsmath,amssymb,amsthm}
\newtheorem{theorem}{\sc Theorem}[section]
\newtheorem{lemma}[theorem]{\sc Lemma}
\newtheorem{proposition}[theorem]{\sc Proposition}

\newcommand{\be}{\begin{equation}}
\newcommand{\ee}{\end{equation}}

\newcommand{\fl}[1]{\lfloor{#1}\rfloor}

\numberwithin{equation}{section} 

\def\bE{\mathbb{E}}
\def\bN{\mathbb{N}}
\def\bP{\mathbb{P}}

\def\bR{\mathbb{R}}
\def\bZ{\mathbb{Z}}

\allowdisplaybreaks[1]
\usepackage{epic} 

\begin{document}

\title[Soft Edge for Longest Increasing Paths ]{Soft edge results for longest increasing paths on the planar lattice}
\author[Nicos ~Georgiou]{Nicos Georgiou}
\address{Nicos Georgiou \\ University of Wisconsin-Madison\\ 
Mathematics Department\\ Van Vleck Hall\\ 480 Lincoln Dr.\\  
Madison WI 53706-1388\\ USA.}
\email{georgiou@math.wisc.edu}
\urladdr{http://www.math.wisc.edu/~georgiou}
\date{\today}
\begin{abstract}
For two-dimensional last-passage time models of weakly increasing paths, interesting scaling limits have been proved for points close the axis (the hard edge). For strictly increasing paths of Bernoulli($p$) marked sites, the relevant boundary is the line $y=px$. We call this the soft edge to contrast it with the hard edge. We prove laws of large numbers for the maximal cardinality of a strictly increasing path in the rectangle $[\fl{p^{-1}n -xn^a}]\times[n]$ as the parameters $a$ and $x$ vary.  The results change qualitatively as $a$ passes through the value $1/2$. 
\end{abstract}
\maketitle

\section{introduction}

\textbf{Basic model.} Consider a collection of independent 
Bernoulli random variables $\{ X_v\}_{v\in\bZ^2}$ with  
$\bP(X_v=1)=p=1-q$ and interpret the event that $X_v=1$ as the event of having site $v$ as \textsl{marked}. For any rectangle $[m]\times[n] = \{ 1,2,...,m \}\times\{1,2,...,n\}$  we can define the random variable $L(m,n)$ that denotes the maximum possible number of marked sites that one can collect along a path  from $(1,1)$ to $(m,n)$ that is
 strictly increasing in both coordinates. It is possible that  there is more than one optimal path, and any such path is called a `Bernoulli longest increasing path (BLIP).' For example in Figure \ref{fig2} a longest increasing path is $\Pi =\{ (1,2), (2,3), (3,4), (5,5), (7,8)\}$.

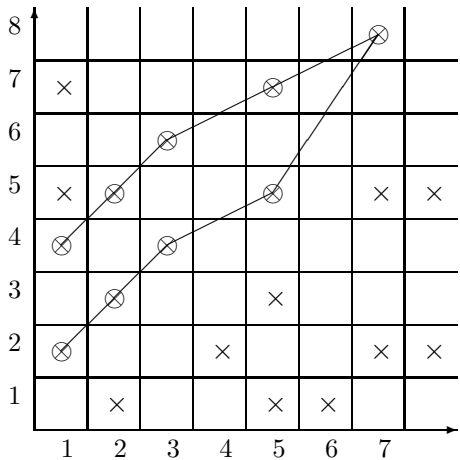
\begin{figure}[ht]
\begin{center}
\begin{picture}(200,160)(0,20)
\put(20,20){\vector(1,0){160}} 
\put(20,20){\vector(0,1){160}} 

\multiput(20,40)(0,20){7}{\line(1,0){160}}
\multiput(40,20)(20,0){7}{\line(0,1){160}} 

\put(30,10){\small 1} \put(50,10){\small 2}
\put(70,10){\small 3} \put(90,10){\small 4}
\put(110,10){\small 5} \put(130,10){\small 6}\put(150,10){\small 7}

\put(10,30){\small 1} \put(10,50){\small 2} \put(10,70){\small 3}
\put(10,90){\small 4} \put(10,110){\small 5}
\put(10,130){\small 6}\put(10,150){\small 7}\put(10,170){\small 8}

\put(26,47){$\otimes$}
\put(26,87){$\otimes$} \put(27,107){$\times$} \put(27,147){$\times$}
\put(47,27){$\times$} \put(107,27){$\times$} \put(127,27){$\times$} 

\put(46,67){$\otimes$} \put(46,107){$\otimes$} 
\put(87,47){$\times$} \put(167,47){$\times$} 

\put(66,87){$\otimes$} \put(66,127){$\otimes$} 
\put(107,67){$\times$} \put(147,47){$\times$} 

\put(106,107){$\otimes$} \put(106,147){$\otimes$} 
\put(147,107){$\times$} \put(167,107){$\times$} 

\put(146,167){$\otimes$}

{\linethickness{3pt}
\put(30,50){\line(1,1){40}} 
\put(70,90){\line(2,1){40}}
\put(110,110){\line(2,3){40}}
}

{\linethickness{3pt}
\put(30,90){\line(1,1){40}} 
\put(70,130){\line(2,1){40}}
\put(110,150){\line(2,1){40}}
}
\end{picture}
\end{center}
\caption{Two possible Bernoulli Longest Increasing paths in the rectangle $[7]\times[8]$. Bernoulli markings are denoted by $\times$. With the notation introduced, we have that $L(7,8) = 5.$\label{fig2}}
\end{figure}

It is easy to see that the random variables $-L(m,n)$ are subadditive. By Kingman's Subadditive Ergodic Theorem and some estimates to take care of integer parts, one can prove $n^{-1}L(\fl{nx},\fl{ny})\rightarrow \Psi(x,y)$ a.s. and in $L^1$. The function $\Psi(x,y)$ was completely determined in \cite{Sepp1}, using the hydrodynamic limit of a certain particle process and it is given by
\be
\Psi(x,y) =\left\{
\begin{array}{lll}

\vspace{0.1 in}
x, & \textrm {if } x < py \\

\vspace{0.1 in}
\displaystyle \frac{2\sqrt{pxy}-p(x+y)}{q}, & \textrm {if } p^{-1}y\geq x\geq py \\

\vspace{0.1 in}
y, &\textrm {if }  y < px 
\end {array}
\right.
\label{Timoiscool} 
\ee for all $(x,y)\in \bR_{+}^2$. There is a vast literature in statistical physics that studies this model as a simplified alternative to the hard \textsl{longest common subsequence} (LCS) model (see for example \cite{StatphysBM}, \cite{Sch}).

\medskip

\textbf{Connection with TASEP.} In \cite{Sch} the authors described a connection between $L(m,n)$ and a discrete totally asymmetric simple exclusion process (DTASEP). This connection converts $L(m,n)$ into a last passage problem of weakly increasing (up-right) paths. We explain this connection rigorously in section 3, in a somewhat simpler
 way than \cite{Sch}. 

To be more precise, we associate independent r.v.'s $Y_v$ to each point $v$ of $\bN^2$ and define the last passage time 
\be G(m,n) = \max_{\pi \in \Pi(m,n)}\sum_{v\in \pi}Y_v, \quad (m,n)\in \bN^2 \label{basicref}\ee where $\Pi(m,n)$ is the collection of all weakly increasing up-right paths in the rectangle $[m]\times[n]$ that start from $(1,1)$ and go up to $(m,n).$ If the start is not $(1,1)$ but a generic site $(k,l)\in \bN^2, k\leq m,l\leq n$, we define
\be G\left((k,l),(m,n)\right) = \max_{\pi \in \Pi\left((k,l),(m,n)\right)}\sum_{v\in \pi}Y_v,\ee with the obvious generalization of $\Pi\left((k,l),(m,n)\right)$ being the collection of all weakly increasing up-right paths in the rectangle $\left([k,m]\times[l,n]\right)\cap \bN^2,$ that start from $(k,l)$ and go up to $(m,n).$

In the case of i.i.d.\ random weights $\{Y_v\}_{v\in \bN^2}$, one can easily check that Kingman's subadditive ergodic theorem (e.g. \cite{Kallenberg} p.192) also applies for the double indexed r.v. \be \xi_{m,n}(x,y)= -G(\fl{nx}-\fl{mx},\fl{ny}-\fl{my})\ee assuming that $\bE{\xi_{0,1}^+}< + \infty$.  Hence, $n^{-1}G([nx],[ny])\longrightarrow \Phi(x,y).$ The function $\Phi$ has been completely determined in the case of i.i.d.\ geometric weights in \cite{Propp},\cite{Propp2} and i.i.d.\ exponential weights in \cite{Rost} (though the author did not use the last passage formulation), while proofs of both results using the hydrodynamic limits can be found in \cite{notonweb}. Both of these cases give
\be
\Phi(x,y) = (x+y)\bE Y_{(1,1)}+ 2\sqrt{\mathbb{V}ar(Y_{(1,1)})xy}.
\ee

\medskip

\textbf{Edge results.} Another interesting question is the behavior of the last passage time when one side of the rectangle is significantly smaller (of different order of magnitude) than the other. For
 i.i.d.\ weights with mean and variance  $1$ and exponential
tails, the following is true:
\begin{theorem}[\cite{GW},\cite{Sepp2-GW}]
Let $0<a<1$. Then, $\dfrac{G(n, \fl{xn^a}) - n}{n^{\frac{1+a}{2}}}\rightarrow 2\sqrt{x}$ in probability.
\label{GlynSepp}
\end{theorem} This was proved in two stages: First, \cite{GW} proved the law of large numbers, namely, for $0<a<1$, $\dfrac{G(n, \fl{xn^a}) - n}{n^{\frac{1+a}{2}}}\rightarrow \alpha\sqrt{x}$ in probability. That $\alpha =2$ was proved in \cite{Sepp2-GW}.   

Then, in \cite{Mar1} (a similar result is obtained in \cite{ba1}) the authors prove a distributional limit for general weights, close to the edge:
\begin{theorem}[\cite{Mar1}]Suppose that $\bE|Y_v|^s < +\infty$, for some $s>2$. Let $\mu = \bE(Y_v)$ and $\sigma^2 = Var(Y_v)$. Then, for all $a$ such that  $0<a<\frac{6}{7}(\frac{1}{2}-s^{-1})$, $$ \frac{G(n, \fl{n^a})-n\mu-2\sigma n^{\frac{1+a}{2}}}{\sigma n^{\frac{3-a}{6}}}\Rightarrow F_{TW}.$$ In particular, if the weight distribution $Y_v$ has finite moments of all orders, then the theorem holds for all $a\in (0,\frac{3}{7})$.
\end{theorem} 
\noindent $F_{TW}$ is the Tracy-Widom distribution of the limiting largest eigenvalue of the GUE.   

We refer to these types of results as \textsl{hard} edge results. They are concerned with properties of the model close to the axis, in a very elongated and thin rectangle and this has an effect on the result. For example, in Theorem \ref{GlynSepp} the centering is exactly the expectation of a horizontal path.

On the other hand, the BLIP model
behaves trivially close to the axes because of the law of large numbers.  See the explanation 
after the theorems in Section \ref{resultsec}, that also indicates
why  $L(\fl{xn},\fl{yn})$ close to $x =yp^{-1}$ is the interesting edge
to consider.  We call this type of edge as the \textit{soft} edge, since the behavior of the model is not 
dictated by a boundary (e.g.\ the $x$-axis), but rather the model itself chose this edge as an appropriate one to change behavior. 
These soft edge
 results are quite different from the hard edge in the sense
 that in order to prove them, one uses central limit theorems.
The proofs depend heavily on the fact that there are many
 \textsl{independent} paths  that can be optimal, with high probability, while if we are restricted close to the axes this is no longer true.  This becomes more obvious by using the connection between the BLIP model and the DTASEP model.

\medskip

\textbf{Connections with some particle processes.} Increasing sequences on the planar lattice were first studied in \cite{Sepp1} using an interacting particle system. The discrete time totally asymmetric exclusion process (DTASEP) described in section 3 is connected to the particle system in \cite{Sepp1}. In \cite{Sepp1}, at time $t=0$ labeled particles start from initial configurations $\{(z_k(0),0)\}_{k\in \bZ}\subseteq \bZ\times\bZ_+ $. At each discrete time step $t$, the particles jump to the left to positions $\{(z_k(t),t)\}_{k\in \bZ}$, where $z_{k-1}(t-1)<z_k(t)\leq z_k(t-1)$ so that there are no Bernoulli marked sites on the segment $$S_k(t) =\{(x,t) : x \in \bZ, z_{k-1}(t-1)< x < z_k(t) \}$$ and $S_k(t)$ is maximal with that property. Notice that the cardinality $|S_k(t)| = \xi(k,t)\wedge(z_{k}(t-1)- z_{k-1}(t-1) - 1)$, where $\bP\{\xi(k,t)=s\} = pq^{s}.$

Suppose now that we start from the same initial configuration of particles $\{(w_k(0),0)\}_{k\in \bZ}$, $w_k(0)=z_k(0)$ for all $k$, but the particles now jump a geometrically distributed distance \textsl{to the right} and the position of the particle $k$ acts now as a block to particle $k-1$ satisfying the rule
\be w_{k-1}(t) = \left( w_{k-1}(t-1) + \tilde{\xi}(k,t)\right)\wedge\left( w_{k}(t-1)\right) \label{DiekerC},\ee where $\tilde{\xi}$ are i.i.d.\ $Geom(q)$, $\bP\{\tilde\xi(k,t)=s\} = pq^{s-1}$. We can couple the processes on the same lattice configuration. Then, at every time step, the site occupation is the same; if site $(x,t)$ was occupied by a particle in the first process, then the same site is occupied in this process and the positions of individual particles satisfy $w_{k-t}(t) = z_k(t).$ The particle process satisfying equation \eqref{DiekerC} is exactly the `geometric jumps with blocking' particle process (case C) that is described in \cite{Dieker} (although in this case the geometric random variables are shifted).  In the DTASEP process, this is precicely the movement of the `holes' that define a platoon.
 
\medskip

\textbf{Organization of this paper.} In section 2 we describe the main results and in section 3 we present the models used in the proofs that follow. In  section 4 we also make the connection between the BLIP model and DTASEP rigorous, in a way different than \cite{Sch}, that only involves induction. We then prove the key lemma
(Proposition \ref{relation})
 that we need for the theorems. In section 5 we prove the main theorems for the BLIP model. 

\medskip 

\textbf{Further notation and conventions.} We denote by $\bN = \{1, 2,...\}$ the set of positive integers and by $\bZ_+ = \{0, 1,... \}$ the set of non-negative integers. Throughout, because of the discrete nature of the jumps of the particle processes, we assume that the jump at time $t$ is completed at time $t$, while it has not yet occurred at time $t-$. We refer to the corner growth model with strictly positive geometric weights as the \textsl{standard} corner growth model. We use the notation $(a,b) \leq (x,y)$ for the partial order in $\bR^2$ for which $a\leq x$ \textsl{and} $b\leq y.$ Finally, as always, $C$ is a constant that changes from line to line.

\medskip

\textbf{Acknowledgments.} I would like to thank my advisor Timo Sepp\"{a}l\"{a}inen for many valuable discussions and suggestions, but most importantly for his patience throughout the preparation of this paper. I also thank two anonymous referees for their useful and constructive comments, for pointing out errors in the original version of this paper and for making the proofs of Proposition \ref{relation} and Theorem \ref{thm23} shorter and more elegant. 

\medskip

\section{Main results}  \label{resultsec}

We present here the main theorems and then comment
on why $[\fl{p^{-1}n - xn^a}] \times [n]$ is the interesting  edge
for the BLIP model. Recall that $L(m,n)$ is the maximum cardinality of Bernoulli marked sites one can collect from a strictly increasing path.

\begin{theorem}
{\rm (a)}
Let $x > 0$, $0 < a \leq \frac{1}{2}$, and $d_n > 0$ any sequence such that $d_n \rightarrow +\infty$ and $d_n = o(n)$. Then $$\frac{n- L(\fl{p^{-1}n-xn^a}, n)}{d_n} \rightarrow 0 \quad \textrm{in probability.}$$ 

{\rm (b)}
Let $x >0 $, $0< a \leq \frac{1}{2}$, and $d_n > 0$, $d_n = o(n),$ any sequence such that $\frac{d_n}{\log n} \rightarrow +\infty$. Then $$\lim_{n\rightarrow \infty}\frac{n- L(\fl{p^{-1}n-xn^a}, n)}{d_n} = 0 \quad a.s.$$ 

\label{basic}
\end{theorem}


\begin{theorem} For $\frac{1}{2}< a <1$ and $x \in \bR_+$, we have the following convergence in probability:
$$\frac{n - L(\fl{p^{-1}n-xn^a},n)}{n^{2a-1}}\longrightarrow\frac{(px)^2}{4q}$$ where $q=1-p$.
\label{thm23}
\end{theorem}

From  (\ref{Timoiscool}) we see that $\Psi(p^{-1},1)=1$. This can be shown as follows. Create a maximal path
by the following patient strategy.
 Start checking the sites  $(1,1), (2,1), (3,1)...$ until the first
marked site $(X_{S_1},1)$. 
Then move up and start checking $(X_{S_1}+1,2),...$ up to the next marked site $(X_{S_2},2)$ and so on.
By the SLLN this strategy gives a path of $n-o(n)$ marked sites
in the rectangle $[\fl{p^{-1}n}]\times[n]$. 
Thus the behavior of 
 $L(m,n)$ is trivial for  $m > \fl{p^{-1}n}$  and this suggests
 $x=p^{-1}y$ as the interesting edge to look at.
Second, in \cite{Sch} they show, through an asymptotic analysis, the following 
\begin{theorem}[\cite{Sch}] For $py < x < p^{-1}y$, we have the following convergence in distribution:
$$\dfrac{q(xy)^{\frac{1}{6}}}{p^{\frac{1}{6}}}\dfrac{L(\fl{nx},\fl{ny}) - \Psi(nx, ny)}{n^{\frac{1}{3}}\left( \sqrt{x} - \sqrt{py} \right)^{\frac{2}{3}}\left( \sqrt{y} - \sqrt{px} \right)^{\frac{2}{3}}}\Longrightarrow F_{TW}.
$$\end{theorem}
Notice that in our case $x=py$; the denominator is $0$ and so we need to treat the edge differently.

One can guess the correct scaling for the edge result by Taylor expanding
 the second branch of $\eqref{Timoiscool}$ 
with $x=p^{-1}n-xn^{a-1}$ and $y=1$. 
This also hints at the cut-off at $a=1/2$.
 Theorems \ref{basic} and \ref{thm23} demonstrate this,
together with the rather surprising vanishing of the error
for $a\leq 1/2$. 

Theorem \ref{basic} can be proved using the same idea as for the proof of Theorem \ref{thm23} but we present a more elementary proof.   

\section{Preliminaries}
\label{Preliminaries}
We begin by describing the DTASEP model. After that we define a new particle process $R$ on the initial lattice configuration of the BLIP model, and show how the DTASEP naturally arises from the process $R$.

\subsection{Discrete time totally asymmetric fragmentation process and discrete TASEP with backward updating}

Consider a one dimensional lattice, where each lattice point is occupied by at most one particle. A string of $n$ consecutive particles is called a platoon (also called a cluster in \cite{Schutz2}) of size $n$, if it is bounded by `holes' (empty sites). The platoon to the left of hole $j$ is the $j-$th platoon and its size is denoted by $n_j$.

\medskip

At each time step, a piece of random size $0\leq M_j \leq n_j$ breaks off from platoon $j$ and moves to the left by one lattice point. 

Define
\begin {displaymath}
\bP(M_{j}= k) =\left\{
\begin{array}{ll}

\vspace{0.1 in}
pq^k, & \textrm {if } k < n_j \\

q^{n_j}, & \textrm {if } k=n_j.
\end {array}
\right.
\end{displaymath}

\medskip

The DTASEP with backward update models the motion viewing the particles as individuals rather than as  fragments of platoons. To be precise,  let $w_i(t)$ denote the position of particle $i$ at time $t$. Start with an initial configuration of particles $w_i = w_i(0)$, satisfying $w_{i-1} < w_i$. At each time step $t$, we update the process from left to right, in the sense that particle $i$ moves one unit to the left at time $t$ with probability $q$ if one of two things is true:

\medskip

(i)  $w_i(t-1) > w_{i-1}(t-1) +1$

\medskip
 
(ii) $w_i(t-1) = w_{i-1}(t-1) +1$ and $w_{i-1}(t) = w_{i-1}(t-1) -1 $

\medskip

In any other case, the jump is suppressed with probability 1. Notice that (ii) implies that if a particle made a jump at time $t$, then the particle immediately to its right has an opportunity to jump, even if the position became available exactly at time $t$.

\medskip

We would like to define DTASEP on the same lattice configuration on which we defined the BLIP model. We need to define a new particle process which will turn out to be important in the proofs that follow.

\medskip

\subsection{The process $R$} First, index each square of $\bN^2$ by its \textsl{lower-}right corner, and then mark each square with $\times$ if the \textsl{upper-}right corner is marked (i.e. the Bernoulli r.v. gets the value $1$ there, which happens with probability $p$). This shifts the lattice $\bN^2$ to the lattice $\bN\times\bZ_+$. Fix a point $(m',n')$ in the new lattice and let $L'(m',n')$ be the cardinality of BLIP if we start from $(1,0)$ and end at $(m',n')$. Then, since we mark each square by the upper-right corner, we have the obvious relation 
\be
L'(m',n') = L(m',n'+1).
\ee
We define the $R$ - process on the two dimensional space $\bN \times\bZ_+$, with time increasing in the vertical direction.

Let $r_k(t)$ be the position of the $k-$th particle of the $R$-process at time $t$. Start with initial particle configuration  \be r_k(0) = k, \quad k\geq 1 .\ee Embed the particles in the lattice $\bZ_+^2$ by defining
\be R_k(t) = (r_k(t), t), \quad k \geq 1, \quad t \in \bZ_+.\ee

\textbf{Evolution of the $R$ process.}

\medskip

There exists (w.p.\ $1$) an particle $k^*$ that  lies on a marked space-time square, such that no particle to its left lies on a marked square ($k^* =5$ in Figure \ref{fig1}). 

At time $t=1$, particle $k^*$ moves $1$ unit to the right, pushing all other particles to its right with it. In the 2-dimensional picture, the particle moves by the vector $(1,1)$. Also notice that platoons start to form. 

In general, for $t>0$, at time $t-1$ we label the platoons from left to right. Each platoon behaves independently. Let $k^*_m$ be the leftmost particle that lies on a marked space-time square in the $m-th$ platoon (i.e. the space-time square $R_{k^*_m}(t-1) = (r_{k^*_m}(t-1),t )$ is marked). Then, at time $t$, particle $k^*_m$ jumps by the vector $(1,1)$ and moves all the particles to its right also, as long as they are in the same platoon with it. The remaining particles of the platoon to the left of $k^*_m$ move by the vector (0,1). It is possible that no particle of a platoon lies on a marked space-time square. If this happens then that whole platoon moves by the vector $(0,1)$.(See Figure \ref{fig1}.) 

To summarize: $r_k(t+1)=r_k(t)+1$ if there exists $k^*\leq k$ such that $r_{k^*}(t)=r_k(t)- (k-k^*)$ (particles $k,k^*$ are in the same platoon) and $R_{k^*}(t)$ is a marked space-time square. In any other case, $r_k(t+1)=r_k(t)$.  

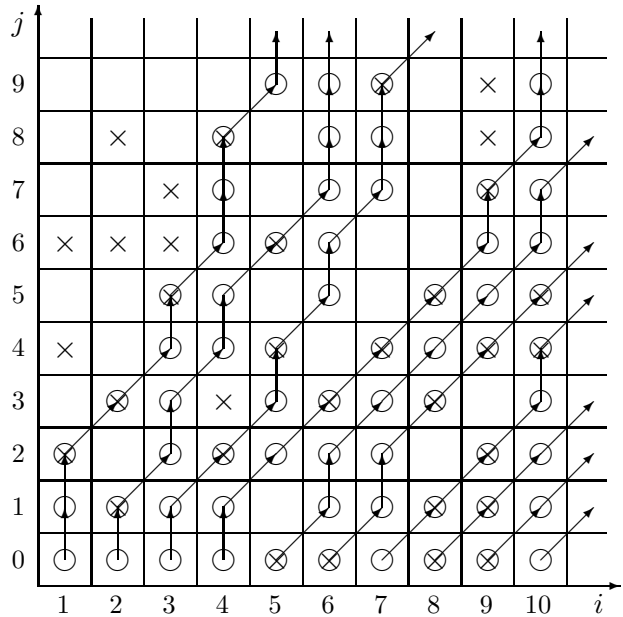
\begin{figure}[ht]
\begin{center}
\begin{picture}(250,250)(-20,-20) 
\put(0,0){\vector(1,0){220}} 
\put(0,0){\vector(0,1){220}} 

\put(7,-10){\small 1} \put(27,-10){\small 2}
\put(47,-10){\small 3} \put(67,-10){\small 4}
\put(87,-10){\small 5} \put(107,-10){\small 6}
\put(127,-10){\small 7}\put(147,-10){\small 8}
\put(167,-10){\small 9}\put(184,-10){\small 10}

\put(-10,7){\small 0} \put(-10,27){\small 1}
\put(-10,47){\small 2} \put(-10,67){\small 3} 
\put(-10,87){\small 4} \put(-10,107){\small 5} 
\put(-10,127){\small 6} \put(-10,147){\small 7} 
\put(-10,167){\small 8} \put(-10,187){\small 9}

\put(210,-10){$i$} \put(-10,210){$j$} 

\multiput(0,20)(0,20){10}{\line(1,0){215}}
\multiput(20,0)(20,0){10}{\line(0,1){215}}

\multiput(10,10)(20,0){4}{\vector(0,1){20}} 
\multiput(90,10)(20,0){6}{\vector(1,1){20}}
\multiput(10,30)(20,0){1}{\vector(0,1){20}}
\multiput(30,30)(20,0){3}{\vector(1,1){20}} 
\multiput(110,30)(20,0){2}{\vector(0,1){20}}
\multiput(150,30)(20,0){3}{\vector(1,1){20}}
\multiput(10,50)(20,0){1}{\vector(1,1){20}}
\multiput(50,50)(20,0){1}{\vector(0,1){20}}
\multiput(70,50)(20,0){4}{\vector(1,1){20}}
\multiput(170,50)(20,0){2}{\vector(1,1){20}}
\multiput(30,70)(20,0){2}{\vector(1,1){20}}
\multiput(90,70)(20,0){1}{\vector(0,1){20}}
\multiput(110,70)(20,0){3}{\vector(1,1){20}}
\multiput(190,70)(20,0){1}{\vector(0,1){20}}
\multiput(50,90)(20,0){2}{\vector(0,1){20}}
\multiput(90,90)(20,0){1}{\vector(1,1){20}}
\multiput(130,90)(20,0){4}{\vector(1,1){20}}
\multiput(50,110)(20,0){2}{\vector(1,1){20}}
\multiput(110,110)(20,0){1}{\vector(0,1){20}}
\multiput(150,110)(20,0){3}{\vector(1,1){20}}
\multiput(70,130)(20,0){1}{\vector(0,1){20}}
\multiput(90,130)(20,0){2}{\vector(1,1){20}}
\multiput(170,130)(20,0){2}{\vector(0,1){20}}
\multiput(70,150)(20,0){1}{\vector(0,1){20}}
\multiput(110,150)(20,0){2}{\vector(0,1){20}}
\multiput(170,150)(20,0){2}{\vector(1,1){20}}
\multiput(70,170)(20,0){1}{\vector(1,1){20}}
\multiput(110,170)(20,0){2}{\vector(0,1){20}}
\multiput(190,170)(20,0){1}{\vector(0,1){20}}
\multiput(90,190)(20,0){2}{\vector(0,1){20}}
\multiput(130,190)(20,0){1}{\vector(1,1){20}}
\multiput(190,190)(20,0){1}{\vector(0,1){20}}


\multiput(10,10)(20,0){10}{\circle{7.5}}
\multiput(85.5,6.5)(20,0){2}{\large$\times$} 
\multiput(145.5,6.5)(20,0){2}{\large$\times$}

\multiput(10,30)(20,0){4}{\circle{7.5}}
\multiput(110,30)(20,0){5}{\circle{7.5}}
\multiput(25.5,26.5)(20,0){1}{\large$\times$} 
\multiput(145.5,26.5)(20,0){2}{\large$\times$}

\multiput(10,50)(20,0){1}{\circle{7.5}}
\multiput(50,50)(20,0){5}{\circle{7.5}}
\multiput(170,50)(20,0){2}{\circle{7.5}}
\multiput(5.5,46.5)(20,0){1}{\large$\times$} 
\multiput(65.5,46.5)(20,0){1}{\large$\times$}
\multiput(165.5,46.5)(20,0){1}{\large$\times$}

\multiput(30,70)(20,0){2}{\circle{7.5}}
\multiput(90,70)(20,0){4}{\circle{7.5}}
\multiput(190,70)(20,0){1}{\circle{7.5}}
\multiput(25.5,66.5)(20,0){1}{\large$\times$} 
\multiput(65.5,66.5)(20,0){1}{\large$\times$}
\multiput(105.5,66.5)(20,0){1}{\large$\times$}
\multiput(145.5,66.5)(20,0){1}{\large$\times$}

\multiput(50,90)(20,0){3}{\circle{7.5}}
\multiput(130,90)(20,0){4}{\circle{7.5}}
\multiput(5.5,86.5)(20,0){1}{\large$\times$} 
\multiput(85.5,86.5)(20,0){1}{\large$\times$}
\multiput(125.5,86.5)(20,0){1}{\large$\times$}
\multiput(165.5,86.5)(20,0){2}{\large$\times$}

\multiput(50,110)(20,0){2}{\circle{7.5}}
\multiput(110,110)(20,0){1}{\circle{7.5}}
\multiput(150,110)(20,0){3}{\circle{7.5}}
\multiput(45.5,106.5)(20,0){1}{\large$\times$} 
\multiput(145.5,106.5)(20,0){1}{\large$\times$}
\multiput(185.5,106.5)(20,0){1}{\large$\times$}

\multiput(70,130)(20,0){3}{\circle{7.5}}
\multiput(170,130)(20,0){2}{\circle{7.5}}
\multiput(5.5,126.5)(20,0){3}{\large$\times$} 
\multiput(85.5,126.5)(20,0){1}{\large$\times$}

\multiput(70,150)(20,0){1}{\circle{7.5}}
\multiput(110,150)(20,0){2}{\circle{7.5}}
\multiput(170,150)(20,0){2}{\circle{7.5}}
\multiput(45.5,146.5)(20,0){1}{\large$\times$} 
\multiput(165.5,146.5)(20,0){1}{\large$\times$}

\multiput(70,170)(20,0){1}{\circle{7.5}}
\multiput(110,170)(20,0){2}{\circle{7.5}}
\multiput(190,170)(20,0){1}{\circle{7.5}}
\multiput(25.5,166.5)(20,0){1}{\large$\times$} 
\multiput(65.5,166.5)(20,0){1}{\large$\times$}
\multiput(165.5,166.5)(20,0){1}{\large$\times$}

\multiput(90,190)(20,0){3}{\circle{7.5}}
\multiput(190,190)(20,0){1}{\circle{7.5}}
\multiput(125.5,186.5)(20,0){1}{\large$\times$} 
\multiput(165.5,186.5)(20,0){1}{\large$\times$}
\end{picture}
\end{center}
\caption{Evolution of the process $R$, in a 10 by 10 rectangle. Circles on any horizontal level $j$ denote particle locations $r_{k}(j)$ at time $j$. The Bernoulli($p$) marks are denoted by $\times$. }\label{fig1}
\end{figure}%

To convert the process $R(t)$ into DTASEP apply the lattice transformation $(i,j)\mapsto(i-j,j)$. Consider the effect on jumps:

\textbf{Case 1:} If particle $k$ in the $R$ process moved by vector $(0,1)$, then in the transformed picture it moves by vector $(-1,1)$. So it takes a jump to the left with probability $q$, independently, as long as there is room to move at that time step (i.e. particle $k-1$ also jumps at the same time or is at a distance greater than 1).

\textbf{Case 2:} If particle $k$ in the $R$ process moved by vector $(1,1)$ with probability $p$, in the transformed lattice it moves by $(0,1)$. All particles in its platoon that it was pushing in the earlier picture are now blocked and they too must move by $(0,1)$ in the transformed picture. 

Let $\tilde{r}_k(t)$ be the position of particle $k$ at time $t$ in DTASEP. By the description above we get \be \tilde{r}_k(t) = r_k(t)-t, \quad t\in \bZ_+.\ee Define \begin{align}\tau(i,k) &= \inf\{t \geq 0: \text{particle $k$ jumped $i$ times in the DTASEP process} \}\notag \\
                              &= \inf\{t \geq 0: \tilde{r}_k(t) = k-i \}\notag \\
                              &= \inf\{t \geq 0: r_k(t) = k-i+t \}.\notag    
\end{align} with boundary conditions $\tau(0,k)= \tau(k,0)=0$ for all $k$. Note that the times $\tau(i,k)$ are strictly increasing in $i$ and non-decreasing in $k$.                       

Observe that \hfill\penalty -10000 
\hbox{\kern 5pt\vbox{
\hbox to 40em{\hfill\hbox to 23em{\textsl{particle $k$ jumped $t-i$ times by time $t$ in DTASEP}}\hfill}
\hbox to 40em{\hfill\hbox to 23em{\textsl{if and only if $k$ jumped $i$ times to the right in $R$.}}\hfill}
}}
\vbox{\kern -50pt\begin{equation}
\label{dtasep-R}
\end{equation}
\kern 50pt
}
For a fixed space-time square $(s,t) \in \bN\times\bZ_+$,
the authors in \cite{Sch} derive the following relation. 

\begin{proposition} Let $(s,t) \in \bN\times \bZ_+$ . Then 
 $$ L'(s,t) = L(s,t+1)= s -  \max\{k: s \geq k\geq (s-t-1)\vee 1,  \tau( t + 1 - s + k , k) \leq t+1\} $$

\noindent with $L'(s,t) = s$ if the above set is empty.
\label{relation}
\end{proposition}

\section{Proof of Proposition \ref{relation}}  

We start with the proof of a preliminary lemma connecting the $R$-process with $L'$, directly followed by the proof that it is in fact equivalent to Proposition \ref{relation}. 

\begin{lemma}
For a space-time square $(s,t) \in \bN\times \bZ_+,$ and $s \geq y\geq 1$, we have that $L'(s,t)\geq y$ if and only if particle $s-y+1$ jumps to the right (in the $R$-process) at least $y$ times during the first $t+1$ time-steps.
\label{insteadofterraces}
\end{lemma}

\begin{proof}
We will proceed by way of induction. 

$(\Longleftarrow)$ Define $P_{n,k} = R(t_{n,k}) = (r(t_{n,k}),t_{n,k})$ to be the position of particle $k$ (in the coordinates of Figure \ref{fig1}) just before it makes its $n$th jump (e.g. in Figure \ref{fig1}, $P_{2,3}=(4,2)$). Our aim is to prove that \be \text{for all $n\geq0 ,k\geq1$ there is a path of weight $n$ to the point $P_{n,k}$}.\label{induction1}\ee 
Suppose that \eqref{induction1} holds; then if particle $s-y+1$ makes $y$ jumps to the right in the first $t+1$ time-steps, we have $P_{y,s-y+1}\leq (s,t)$ and so there is in fact a path of weight at least $y$ to $(s,t)$ as desired for the lemma.

\textsl{Base Case:} For $k=1$ its easy to verify \eqref{induction1} for all $n$, since the first particle can only jump when it lands on a marked site. For the $n = 0$ case and arbitrary $k$, first assume $t_{1,k}\geq1$ and then observe that if particle $k$ did not jump in the first $t_{1,k}$ time steps, then the rectangle $[1,k]\times[0,t_{1,k}-1]$ has no marked sites and all paths in it have weight $0$. If $t_{1,k}=0$ then it does not make sense to consider $n=0$ since there exists a mark at $(k',0)\leq (k,0)$ for some minimal $k'$ and $P_{1,k}=1$ for $k \geq k'.$ In either case though, the induction base case is true.

\textsl{Induction step:} As explained in Section \ref{Preliminaries}, the $n$th jump of particle $k$ happens either because $P_{n,k}$ is marked (so the particle `decides' to jump itself) or because it is pushed to the right by some particle $k'< k$ which is itself performing its $n$th jump, in which case there is a mark at site $P_{n,k'}.$ In either case there is a mark at site $P_{n,k'}$ for some $k'\leq k.$ Now consider the point $P_{n-1,k'}$. This point is strictly south-west of $P_{n,k'}.$ By the induction hypothesis there is a path of weight $n-1$ to this point. Adding the mark at point $P_{n,k'}$ gives a path of weight $n$ to the point $P_{n,k'}$ and since $P_{n,k'}\leq P_{n,k}$, this is also a path of weight $n$ to the point $P_{n,k}.$   

$(\Longrightarrow)$ Suppose there is a path of weight $y$ to the point $(s,t).$ Let the marks on this be $(i_0,j_0), (i_1,j_1),...,(i_{y-1},j_{y-1}).$ By definition, $1\leq i_0 < i_1 <...<i_{y-1}\leq s$ and $0\leq j_0<j_1<...<j_{y-1}\leq t.$ 

We aim to prove that 
particle $i_r-r$ has jumped at least $r+1$ times during the first $j_r+1$ steps of the $R$-process (hence also all particles to the right of $i_r-r$ have jumped at least $r+1$ times during the first $j_r+1$ steps). Note that at time $n$, the $n$th time step is completed.

\textsl{Base Case:} For $r=0$ we want to show that particle $i_0$ jumped at least once in the first $j_0+1$ time-steps. We know that there is a mark at site $(i_0, j_0)$ so, either particle $i_0$ touches that site because it did not jump earlier, or it has already jumped at an earlier time. In any case, the statement is true.

\textsl{Induction step:} Suppose it is true for $r=r_0-1$ and we wish to show it for $r=r_0.$ Note that $i_{r_0}-r_0\geq i_{r_0-1}-(r_0-1)$. By the induction hypothesis, particle $i_{r_0}-r_0$ has jumped at least $r_0$ times \textsl{before} time $j_{r_0}$ (because particle $i_{r_0}-r_0$ is the same particle as, or to the right of, particle $i_{r_0-1}-(r_0-1)$).

If in fact particle $i_{r_0}-r_0$ has already jumped $r_0+1$ times before time $j_{r_0}$, then we are done. Otherwise, it has jumped precisely $r_0$ times before time $j_{r_0}$. In this case, the particle is at position $i_{r_0}$ at time $j_{r_0}-1$ and it will jump at time $j_{r_0}$ due to the mark at point $(i_{r_0},j_{r_0})$. So indeed the particle has jumped at least $r_0+1$ when time-step $j_{r_0}+1$ is completed. 

In particular, putting $r=y-1$ shows that particle $i_{y-1}-y+1$ has jumped $y$ times during the first $j_{y-1}+1$ time-steps. Since $i_{y-1}\leq s$ and $j_{y-1}\leq t$, this implies that particle $s-y+1$ has jumped at least $y$ times during the first $t+1$ time-steps, as required by the lemma.
\end{proof}

\begin{proof}[Proof of Proposition \ref{relation}]
Let $L'(s,t) = n$ for some $n \in \bZ_+, n\leq \min\{s,t+1\}.$ Assume first that $\mathcal{C} =\{k: s \geq k\geq (s-t-1)\vee 1,  \tau( t + 1 - s + k , k) \leq t+1\}\neq \varnothing.$
For $(s,t)\in \bN\times\bZ_+$, define
\be k_{s,t}^* = \max\{k: s \geq k\geq (s-t-1)\vee 1,  \tau( t + 1 - s + k , k) \leq t+1\}.\ee 
We are going to show that 
\be
n=s-k_{s,t}^*.
\ee
By Lemma \ref{insteadofterraces} we know that particle $s-n+1$ jumped at least $n$ times by time $t+1$. By \eqref{dtasep-R} we have that it jumped at most $t-n+1$ times to the left in DTASEP, by time $t+1$. Hence, 
\be
\tau(t+1-s+(s-n+1),s-n+1) =\tau((t-n+1)+1,s-n+1)> t+1.  
\ee This implies that $k_{s,t}^* < s-n+1$; equivalently
\be
n\leq s -k_{s,t}^*. 
\ee

For the other inequality, observe that particle $s-n$ jumped to the right at most $n$ times in the $R$-process (in the opposite case Lemma \ref{insteadofterraces} implies that $L'(s,t)\geq n+1$ which contradicts our hypothesis). So particle $s-n$ jumped at least $t+1-n$ times in DTASEP by time $t+1$, therefore,
\be
\tau(t+1-s+(s-n),s-n) =\tau(t+1 -n,s-n) \leq t+1.  
\ee
This implies $k_{s,t}^* \geq s-n$; equivalently 
\be
s -k_{s,t}^*\leq n. 
\ee
To finish the proof, consider the case where $\mathcal{C} = \varnothing$. Let $k_0 = (s-t-1)\vee 1.$ Since $\tau(t+1-s+k,k)$ is non decreasing in $k$, $\mathcal{C} = \varnothing $ if and only if $s\leq t+1$ and $\tau(t+1-s+ k_0, k_0)\geq t+2$. This implies that $k_0=1$ and that particle $1$ jumped at least $s$ times to the right by time $t+1$, in the $R$ process (since $\tau(t-s+2,1)\geq t+2$ implies that by time $t+1$ the first particle jumped at most $t-s+1$ times in DTASEP). By Lemma \ref{insteadofterraces}, $L'(s,t)\geq s.$
\end{proof}

\section{Proof of results for the Longest Increasing Path model}

To make the notation slightly simpler, we can convert back to $L(m,n).$ Let $(m,n)\in \bN^2.$ Since $L(m,n)= L'(m,n-1)$, by Proposition \ref{relation} we get the equivalent form 
\be 
L(m,n) = m - \displaystyle\left(\max\{ k: (m-n)\vee 1\leq k\leq m, \tau( n- m + k  ,k) \leq n\}\vee 0 \right)
\label{proofformula}
\ee

Set $k^* = \max\{(m-n)\vee 1\leq k\leq m : \tau( n-m + k ,k) \leq n\}\vee 0 .$ For $(m-n)\vee 1\leq j\leq m$, we have the equality of events
\be
\mathcal{B}_{m,n,j}=\left\{L(m,n)\leq m-j\right\}=\left\{j\leq k^*\right\}= \left\{\tau( n-m + j ,j) \leq n\right\}
\label{defB}   
\ee
where the second equality comes from Proposition \ref{relation} and the last equality comes from 
the fact that $\tau(n-m+ \cdot, \cdot)$ is non-decreasing. It is going to be notationally convenient for the proofs that follow, to allow non-integer arguments in $\tau(n-m+ \cdot, \cdot).$ For $j \geq 1, j\notin \bN$, define
\be
\tau(n-m+ j, j) = \tau(n-m+ \fl{j}, \fl{j})
\ee and extend the definition of $\mathcal{B}_{m,n,j}$ in the obvious way.

A distributionally equivalent way of defining the process $\{\tau(i,j)\}_{i,j\geq 1}$ is by using the recursion 
\be 
\tau(i,j) = \left(\tau(i-1,j)+ 1\right)\vee\tau(i,j-1)+ \widetilde{Y}_{ij}\label{connection}
\ee 
where $\bP(\widetilde{Y}_{ij} = s) = qp^s$, $s\geq 0$ and $\{\tilde{Y}_{ij}\}$ are i.i.d.\  for $i, j\geq 1.$   

In words, the time that particle $j$ performs its $i$-th jump cannot happen before two events occur. First, particle $j$ itself needs to  jump $i-1$ times and is allowed to jump again starting from the next time step. Second, particle $j-1$ needs to jump $i$ times or else the exclusion rule forbids $j$ to jump so many times. The updating allows $j$ to jump its $i$th jump exactly at time $\left(\tau(i-1,j)+ 1\right)\vee\tau(i,j-1)$ with probability $q$. After these events occur, particle $j$ waits a geometrically distributed time for its next jump. 

We can connect equation \eqref{connection} with last passage times of the standard corner growth model with geometric weights.
\begin{lemma}Let $i\geq1$, $j\geq1$ and let $G(i,j)$ be defined by \eqref{basicref} with geometrically distributed random weights  $Y_v$, $\bP(Y_v = s) = qp^{s-1},$ for $s\in\bN$. Then,
\begin{equation}
\tau(i,j) =_{\mathcal{D}} G(i,j) - j + 1. 
\label{Timo}
\end{equation}
\end{lemma}

\begin{proof}
Recall that $Y_v =\widetilde{Y}_v +1 $. We begin by showing that
\be
\tau(i,j) = i + \max_{\pi \in \Pi(i,j)}\sum_{v\in \pi}\widetilde{Y}_v.
\label{lptau}
\ee
We induct on $n = i+j.$

\textsl{Base Case:} If $n=2$ then $i=j=1$ and a comparison between \eqref{connection} and \eqref{lptau} proves the base case (recall that $\tau(i,0) = \tau(0,j) = 0$).

\textsl{Induction Step:} Assume $n\geq 3$ and that \eqref{lptau} is true for all $i+j = n-1.$ We are going to show it for $i+j =n.$
\begin{align}
\tau(i,j) &= \left(\tau(i-1,j)+ 1\right)\vee\tau(i,j-1)+ \widetilde{Y}_{ij}\notag \\
          &=\left(i+\max_{\pi \in \Pi(i-1,j)}\sum_{v\in \pi}\widetilde{Y}_v \right)\vee\left(i+\max_{\pi \in \Pi(i,j-1)}\sum_{v\in \pi}\widetilde{Y}_v \right)+ \widetilde{Y}_{ij} \notag \\
          &= i+ \max_{\pi \in \Pi(i,j)}\sum_{v\in \pi}\widetilde{Y}_v. \notag
\end{align}
Now observe that on any up-right path we have exactly $i+j-1$ vertices. Then we can write \eqref{lptau} as 
\small{
\be
\tau(i,j)=  \max_{\pi \in \Pi(i,j)}\sum_{v\in \pi}\left\{\widetilde{Y}_v+1\right\} -j + 1=  \max_{\pi \in \Pi(i,j)}\sum_{v\in \pi}Y_v -j + 1 =  G(i,j) -j+1.\notag \qed \qedhere
\ee
}
\end{proof}

Now, to prove the main theorems.
\begin{proof}[Proof of Theorem \ref{basic}]
Proof of part (a). Let $\epsilon>0$ and let $d_n$ be a positive sequence such that $d_n\longrightarrow +\infty$, with $d_n = o(n).$ We want to show that for all $\epsilon >0$, 
\be
\lim_{n\rightarrow +\infty}\bP\left\{ \epsilon \leq \dfrac{n-L(\fl{p^{-1}n-xn^a},n)}{d_n} \right\}= 0.
\label{wish}
\ee 
Define $m= m(n) = \fl{p^{-1}n-xn^a}$ and $j= j(n) = m - n+\epsilon d_n.$ Notice that for $n$ large enough, $(m-n)\vee1\leq j \leq m.$ Therefore, we can rewrite equation \eqref{wish} using equation \eqref{defB}, and so it is equivalent to prove
\be
\lim_{n\rightarrow +\infty}\bP\left\{\mathcal{B}_{m,n,j} \right\}=\lim_{n\rightarrow +\infty}\bP\left\{ \tau( n-m + j ,j) \leq n\right\} = 0.
\label{wish2}
\ee 
From the definition of $j$ and equation \eqref{Timo}, we get 
\begin{align}
\bP\left\{ \tau( n-m + j ,j) \leq n\right\}&=\bP\left\{ G(\epsilon d_n, m-n+\epsilon d_n)\leq m+\epsilon d_n-1\right\}.
\end{align}
In order to prove \eqref{wish2}, we are going to show that 
\be
\lim_{n\rightarrow +\infty} \bP\left\{ G(\epsilon d_n, m-n+\epsilon d_n)\leq m+\epsilon d_n-1\right\}=0.
\label{wish3}
\ee 
Consider the rectangle $[\fl{\epsilon d_n}]\times[\fl{j}]$. Define $$\pi_i  = \{(1,1),(2,1),...,(i,1)\}\cup\{ (i,2),(i,3),.....,(i,\fl{j})\}\cup\{(i+1, \fl{j}),... (\fl{\epsilon d_n}, \fl{j})\}.$$ Also, set $$S_{\pi_i} =\sum_{v\in\pi_i}Y_v \quad\textrm{and} \quad  S_{(i,j)} = \sum_{k=1}^{\fl{j}}Y_{ik}.$$ For $c\in \bR$, we have the inclusion of events:

\be 
\{ G(\epsilon d_n, j) \leq c\}\subseteq \bigcap_{i \leq \epsilon d_n } \{S_{\pi_i} \leq c \}\subseteq \bigcap_{i \leq \epsilon d_n} \{S_{(i, j)}\leq c-\fl{\epsilon d_n}+1\}
\label{pathbound}
\ee where the last inclusion follows from the fact that the geometric weights $Y_{ik}$ start from $1$. Recall that $\bE Y_{ik} = q^{-1}.$ Note that $\bE S_{(i, j)} = p^{-1}n - xq^{-1}n^a + \epsilon q^{-1}d_n +Cq^{-1}$, where $C < 0$ is the error coming from the integer parts. Beeing a bit careful with the integer parts, we estimate
\begin{align} 
\bP\left\{\mathcal{B}_{m,n,j}\right\} &\leq \left(\bP\left\{S_{(i,j)} \leq \fl{p^{-1}n - xn^a}\right\}\right)^{\fl{\epsilon d_n}} = \left(\bP\left\{S_{(i,j)} \leq p^{-1}n - xn^a\right\}\right)^{\fl{\epsilon d_n}} \notag \\
                                      &= \left(\bP\left\{S_{(i,j)} - \bE S_{(i,j)}\leq p^{-1}n - xn^a- \bE S_{(i,j)}\right\}\right)^{\fl{\epsilon d_n}}.\label{Est}  
\end{align}
Since we are assuming that $ a \leq 1/2$ and $d_n > 0$, there exist $\delta > 0$ and $n_0 = n_0 (\delta) < +\infty$, such that for all $n > n_0$ we have
\be
\bP\left\{S_{(i,j)} - \bE S_{(i,j)}\leq  x\frac{p}{q}n^a -\frac{\epsilon}{q}d_n  \right\} <1 - \delta 
\ee 
by virtue of the CLT. Combining this with \eqref{Est}, we have proved equation \eqref{wish3} and thereby part (a) of Theorem \ref{basic}.

For part (b). Observe that in part (a) we actually proved that for $\epsilon >0$ and $n$ large enough, we have
\be
\bP\left\{ \epsilon \leq \frac{n- L(\fl{p^{-1}n-xn^a}, n)}{d_n}\right\} \leq \left( 1-\delta \right)^{{\epsilon d_n}}.
\ee
A Borel-Cantelli argument finishes the proof.
\end{proof}

Before proceeding to the proof for the non-trivial edge, we need some preliminary comments. We are going to use a modified version of Theorem \ref{GlynSepp} as shown in the next Lemma  (for which we omit the proof).
\begin{lemma}
Let $\mu$ be the expectation and $\sigma^2$ the variance of the weights $Y_{ik}$. Assume that $j/n \rightarrow c_1$, as $n \rightarrow \infty$,  $0<c_1<+\infty$, $0< y < +\infty$ are constants and $0<\beta< 1$. Then,  
\be
G(j, yn^{\beta})= \mu j + n^{\frac{1+\beta}{2}}(2\sigma \sqrt{c_1y} +o(1)) \quad \textrm{in probability,}
\label{Martin}
\ee
where $o(1)$ is a quantity that goes to $0$ in probability as $n$ gets large. 
\end{lemma}

We are going to apply \eqref{Martin} in the case of $\beta=2a-1.$

\begin{proof}[Proof of Theorem \ref{thm23}]
Recall that now $1> a > \dfrac{1}{2}.$ Let $c>0$ be a constant to be specified later and set $ m = \fl{np^{-1}- xn^a}$ and $j=m - n+ \fl{(cn)^{2a-1}}.$ Also let $\mu = \dfrac{1}{q}$ the mean and $\sigma = \frac{\sqrt{p}}{q}$ to be the standard deviation of the geometric weights. 

From equation \eqref{defB} we have \be\mathcal{B}_{m,n,j} =\{ L(\fl{p^{-1}n-xn^a},n) \leq n - \fl{(cn)^{2a-1}}\}.\label{mdnght}\ee
Using \eqref{defB} and \eqref{Timo}, we evaluate
\begin{align}
\displaystyle \bP\left\{\mathcal{B}_{m,n,j}\right\} &=\bP\left\{ G(\fl{(cn)^{2a-1}},j) \leq m +\fl{(cn)^{2a-1}} -1\right\} \notag \\ 
                                                    &=\bP\left\{ G(j,\fl{(cn)^{2a-1}}) \leq m +\fl{(cn)^{2a-1}} -1\right\} \label{Refprob}
\end{align} where the second equality follows from  the distributional equality $G(x,y)=_{\mathcal{D}}G(y,x).$ 
Set $\beta = 2a-1 $ and $y= c^{2a-1}=c^{\beta}.$ Then \eqref{Refprob} becomes
\be
\bP\left\{\mathcal{B}_{m,n,j}\right\}=\bP\left\{ G(j,\fl{yn^{\beta}}) \leq \fl{np^{-1}- xn^{\frac{1+\beta}{2}}}+\fl{yn^{\beta}} -1\right\}.
\label{finalprob}
\ee

Observe that $j/n \rightarrow q/p$. Substituting this in \eqref{Martin}, we get the equality in probability
\be
G(j, yn^{\beta}) = p^{-1}n - n^{\frac{1+\beta}{2}}\left(\frac{1}{q}(x-2\sqrt{qy}) +o(1)\right)
\label{defU}
\ee 

Now compare the expression in the probability of \eqref{finalprob} with \eqref{defU}, keeping in mind that $\beta < \frac{1+\beta}{2}.$
We conclude that 
\be
\lim_{n\rightarrow +\infty}\bP\left\{\mathcal{B}_{m,n,j}\right\} = 0  
\ee
if $x > \dfrac{1}{q}(x-2\sqrt{qy}),$ which is equivalent to $y > \frac{(px)^2}{4q}$ as desired. Similarly, if $x < \dfrac{1}{q}(x-2\sqrt{qy})$ 
\be
\lim_{n\rightarrow +\infty}\bP\left\{\mathcal{B}_{m,n,j}\right\} = 1  
\ee
and this gives the other direction.  
\end{proof}


\begin{thebibliography}{10}

\bibitem{ba1}
Jinho Baik and Toufic~M. Suidan.
\newblock A {GUE} central limit theorem and universality of directed first and
  last passage site percolation.
\newblock {\em Int. Math. Res. Not.}, (6):325--337, 2005.

\bibitem{Mar1}
Thierry Bodineau and James Martin.
\newblock A universality property for last-passage percolation paths close to
  the axis.
\newblock {\em Electron. Comm. Probab.}, 10:105--112 (electronic), 2005.

\bibitem{Propp}
Henry Cohn, Noam Elkies, and James Propp.
\newblock Local statistics for random domino tilings of the aztec diamond.
\newblock {\em Duke Math. J.}, 85(1):117--166, 1996.

\bibitem{Dieker}
A.B. Dieker and J.~Warren.
\newblock Determinental transition kernels for some interacting particles on
  the line.
\newblock {\em Ann. Inst. H. Poin. Probab. Statist.}, 44(6):1162--1172, 2008.

\bibitem{GW}
Peter~W. Glynn and Ward Whitt.
\newblock Departures from many queues in series.
\newblock {\em Ann. Appl. Probab.}, 1(4):546--572, 1991.

\bibitem{Propp2}
William Jockusch, James Propp, and Peter Shor.
\newblock Random domino tilings and the arctic circle theorem.
\newblock {\em arXiv:math/9801068}.

\bibitem{Kallenberg}
Olav Kallenberg.
\newblock {\em Foundations of modern probability}.
\newblock Probability and its Applications (New York). Springer-Verlag, New
  York, second edition, 2002.

\bibitem{StatphysBM}
Satya~N. Majumdar, Kirone Mallick, and Sergei Nechaev.
\newblock Bethe ansatz in the {B}ernoulli matching model of random sequence
  alignment.
\newblock {\em Phys. Rev. E (3)}, 77(1):011110, 10, 2008.

\bibitem{Sch}
V.B. Priezzhev and G.M. Sch\"utz.
\newblock Exact solution of the {B}ernoulli matching model of sequence
  alignment.
\newblock {\em J. Stat. Mech.}, 2008, P09007 (electronic).

\bibitem{Schutz2}
A.~R{\'a}kos and G.~M. Sch{\"u}tz.
\newblock Current distribution and random matrix ensembles for an integrable
  asymmetric fragmentation process.
\newblock {\em J. Stat. Phys.}, 118(3-4):511--530, 2005.

\bibitem{Rost}
H.~Rost.
\newblock Nonequilibrium behaviour of a many particle process: density profile
  and local equilibria.
\newblock {\em Z. Wahrsch. Verw. Gebiete}, 58(1):41--53, 1981.

\bibitem{Sepp1}
Timo Sepp{\"a}l{\"a}inen.
\newblock Increasing sequences of independent points on the planar lattice.
\newblock {\em Ann. Appl. Probab.}, 7(4):886--898, 1997.

\bibitem{Sepp2-GW}
Timo Sepp{\"a}l{\"a}inen.
\newblock A scaling limit for queues in series.
\newblock {\em Ann. Appl. Probab.}, 7(4):855--872, 1997.

\bibitem{notonweb}
Timo Sepp{\"a}l{\"a}inen.
\newblock Hydrodynamic scaling, convex duality and asymptotic shapes of growth
  models.
\newblock {\em Markov Process. Related Fields}, 4(1):1--26, 1998.

\end{thebibliography}

\end{document}